\documentclass[reqno]{amsart}
\usepackage{amsmath,amsthm,amsfonts,amssymb,srcltx}
\usepackage{graphicx} 
\usepackage{epstopdf}
\usepackage{enumerate}

\newtheorem{thm}{Theorem}
\newtheorem{prop}{Proposition}
\newtheorem{cor}{Corollary}

\theoremstyle{remark}
\newtheorem{rem}{Remark}

\newcommand{\Nd}{\mathcal{N}}
\newcommand{\bb}{\pmb{b}}
\newcommand{\nb}{\pmb{n}}
\newcommand{\eb}{\pmb{e}}
\newcommand{\HN}{\mathcal{H}_N}

\begin{document}
\date{}

\author{N.\ V.\ Alexeev}
\address{Chebyshev Laboratory\\ St. Petersburg State University\\
14th Line V.O. 29B\\
St. Petersburg 199178 Russia}
\email{nikita.v.alexeev{\char'100}gmail.com}

\author{J.\ E.\ Andersen}
\address{Centre for Quantum Geometry of Moduli Spaces\\Department of 
Mathematics\\
University of Aarhus\\
DK-8000 Aarhus C, Denmark}
\email{andersen{\char'100}qgm.au.dk}

\author{R.\ C.\ Penner}
\address{Centre for Quantum Geometry of Moduli Spaces\\Department of 
Mathematics\\
University of Aarhus\\
DK-8000 Aarhus C, Denmark\\
and Mathematics Department, Caltech\\
Pasadena, CA 91125 USA}
\email{rpenner{\char'100}caltech.edu}

\author{P.\ G.\ Zograf}
\address{St.Petersburg Department \\
Steklov Mathematical Institute\\
Fontanka 27\\ St. Petersburg 191023, and
Chebyshev Laboratory\\ St. Petersburg State University\\
14th Line V.O. 29B\\
St.Petersburg 199178 Russia}
\email{zograf{\char'100}pdmi.ras.ru}

\title  [Enumerating chord diagrams]{Enumeration of chord diagrams on many 
intervals and their non-orientable analogs}

\thanks{Research supported by the Centre for Quantum Geometry of Moduli Spaces 
which is funded by the Danish National Research Foundation.
The work of NA and PZ was further supported by the Government of the Russian 
Federation megagrant 11.G34.31.0026, by JSC ``Gazprom Neft", and by the RFBR 
grants 13-01-00935-a and 13-01-12422-OFI-M2. In addition, NA was partially 
supported by the RFBR grant 14-01-00500-a and by the SPbSU grant 6.38.672.2013, 
and PZ  was partially supported by the RFBR grant 14-01-00373-a}

\begin{abstract}
Two types of connected chord diagrams with
chord endpoints lying in a collection of
ordered and oriented real segments are considered here:
the real segments may contain additional 
bivalent vertices in one model but not in the other.   
In the former case, we record in a generating function the number of fatgraph
boundary cycles containing a fixed number of bivalent vertices
while in the latter, we instead record the number of boundary
cycles of each fixed length. 
Second order, non-linear, algebraic partial differential
equations are derived which are satisfied by these generating functions in each 
case 
giving efficient enumerative schemes. Moreover, these generating functions 
provide multi-parameter families of solutions 
to the KP hierarchy. For each model, there is furthermore a non-orientable 
analog, and each such model likewise has its own associated differential 
equation. The enumerative problems we solve are interpreted in terms of certain 
polygon gluings. 
As specific applications, we
discuss models of several interacting RNA molecules.
We also study a matrix integral which  computes numbers of 
chord diagrams in both orientable and
non-orientable cases in the model with bivalent vertices, and
the large-N limit is computed using techniques of free probability.
\end{abstract}

\maketitle

\section{Introduction} 
A {\em partial chord diagram} is a connected fatgraph (i.e., a graph equipped 
with a cyclic order on the half edges incident to each vertex)
comprised of an ordered set of $b\geq 1$ disjoint real line segments (called 
{\em backbones}) connected with $k\geq 0$ {\em chords} in the upper half plane 
with distinct endpoints, so that there are $2k$ vertices of degree three (or 
{\em chord endpoints}) and $l\geq 0$ vertices of degree two (or {\em marked 
points}) all belonging to the backbones (in effect, ignoring the vertices of 
degree one arising from backbone endpoints.) If $l=0$ so there are no marked 
points, then we call the diagram a {\em (complete) chord diagram}. Each partial 
or complete chord diagram is a spine of an orientable surface with $n\geq 1$ 
boundary components and therefore has a well-defined topological genus. The 
genus $g$ of a partial chord diagram on $b$ backbones and its number $n$ of 
boundary components are related by Euler's formula $b-k+n=2-2g$.

Chord diagrams occur pervasively in mathematics, which further highlights the 
importance of the counting results obtained here. To mention a few, see the 
theory of finite type invariants of knots and links \cite{K}
(cf. also \cite{BN1}), the representation theory of Lie algebras \cite{CM}, the 
geometry of moduli spaces of flat connections on surfaces \cite{AMR1,AMR2} and 
mapping class groups \cite{ABMP}.  Moreover and as we shall further explain 
later,  partial and complete chord diagrams each provide a useful model  
\cite{PW,P4,OZ,VOZ} for the combinatorics of interacting RNA molecules with the 
associated genus filtration of utility in enumerative problems  
\cite{ACPRS,Bon,OZ,POZ,PTOZ,VOZ,VROZ} and in folding algorithms on one 
\cite{gfold,TT2NE} and two backbones \cite{AHPR}.

Our goal is to enumerate various classes of connected partial and complete chord 
diagrams, and to this end, we next introduce
combinatorial parameters, where each enumerative problem turns out to be solved 
by an elegant partial differential equation 
on a suitable generating function in dual variables.
In effect, creation and annihilation operators for the combinatorial data are 
given by multiplication and differentiation operators in the dual variables 
leading to algebraic differential equations.

We say that a partial chord diagram has 
\begin{itemize}
\item \textit{backbone  spectrum} $\bb=(b_1,b_2,\ldots)$ if the diagram has 
$b_i$ backbones with precisely $i\geq 1$ vertices (of  degree either two or 
three);

\item \textit{boundary point spectrum} $\nb=(n_0,n_1,\ldots)$ if its boundary 
contains $n_i$ connected components with $i$ marked points;

\item \textit{boundary length spectrum} ${\boldsymbol{p}}=({p}_1,{p}_2,\ldots)$ 
if the boundary cycles of the diagram
consist of ${p}_i$ edge-paths of length $i\geq 1$, where the {\em length} of a 
boundary cycle is the number of chords
it traverses counted with multiplicity  (as usual on the graph obtained from the 
diagram by collapsing each backbone to a distinct point) {\sl plus} the number 
of backbone undersides it traverses (or in other words, the number of traversed 
backbone intervals obtained by removing all the chord endpoints from all the 
backbones).
\end{itemize}

The data $\{g,k,l;\bb;\nb;\boldsymbol{p}\}$ is called the {\em type} of a 
partial chord diagram (cf. Fig.~\ref{cd_example}). Note that the entries in the 
data set $\{g,k,l;\bb;\nb;\boldsymbol{p}\}$ are not independent. In particular, 
we have 
\begin{align*}
&b=\sum_{i\geq 1} b_i, \quad n=\sum_{i\geq 0} n_i=\sum_{i\geq 1} {p}_i,\quad 
l=\sum_{i>0}in_i,\\
&2k+l=\sum_{i>0} ib_i,\quad 2k+b=\sum_{i\geq 1}i{p}_i.
\end{align*}
Let $\Nd_{g,k,l}({\bb},{\nb},\boldsymbol{p})$ denote the number of distinct 
connected partial chord diagrams of type $\{g,k,l;{\bb};{\nb};\boldsymbol{p}\}$ 
taken to be zero if there are no chord diagrams of the specified type.  Our two 
basic models involve boundary point spectra of partial chord diagrams and 
boundary length spectra of complete chord diagrams, and each basic model, in 
turn, has both an orientable and a non-orientable incarnation.

\begin{figure}[hbt]%
 \begin{center}
 \includegraphics[width=9cm]{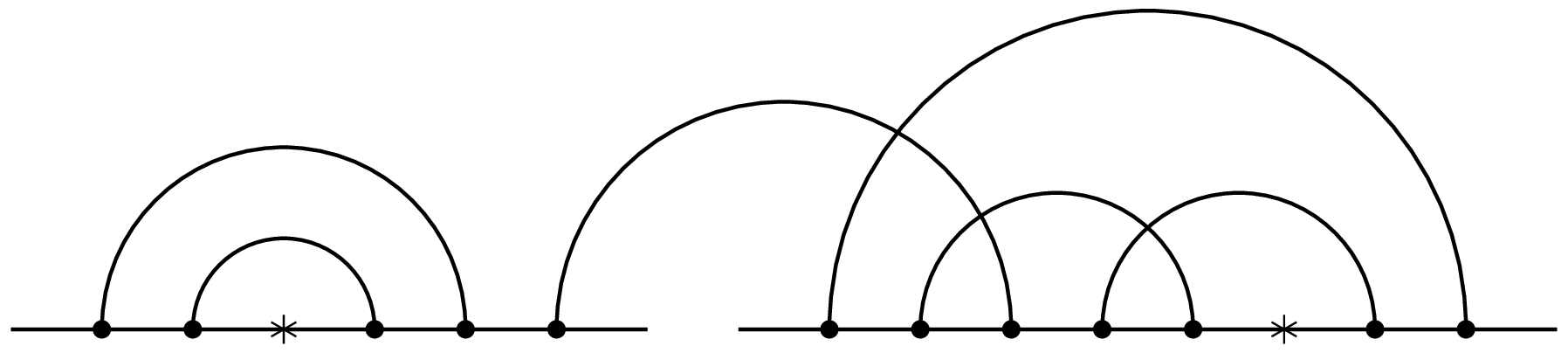}
 \caption{The partial chord diagram of type 
$\{1,6,2;\pmb{e}_6+\pmb{e}_8;2\pmb{e}_0+2\pmb{e}_1; 
\pmb{e}_1+2\pmb{e}_2+\pmb{e}_9\}$. 
Here $\pmb{e}_i$ stands for the sequence with 1 in the $i$-th place and $0$ 
elsewhere.}
 \label{cd_example}%
\end{center}
\end{figure}

We may also consider non-orientable chord diagrams. Let 
$\widetilde\Nd_{h,k,l}(\bb,\nb,\boldsymbol{p})$
denote the number of both orientable and non-orientable connected diagrams on 
$b$ backbones, out of which exactly $b_i$ have $i$ vertices, with $k$ pairs of 
vertices connected by (twisted or untwisted) chords, with boundary point and 
boundary length spectra $\nb$ and $\boldsymbol{p}$ respectively, and with
Euler characteristic $2-h-n,\; n=\sum_{i=1}^\infty n_i$, where $h$ denotes twice 
the genus in the orientable case and the number of cross caps in the 
non-orientable case. This can evidently be formalized in the language of planar 
projections of chord diagrams by two-coloring the chords depending upon whether 
they preserve or reverse the orientation of the plane of projection.

For partial chord diagrams and boundary point spectra, we shall count the 
subsets
$$\Nd_{g,k,l}(\bb,\nb)=\sum_{\boldsymbol{p}} 
~\Nd_{g,k,l}(\bb,\nb,\boldsymbol{p})$$
in the orientable case and
$$\widetilde\Nd_{h,k,l}(\bb,\nb)=\sum_{\boldsymbol{p}}~ 
\widetilde\Nd_{h,k,l}(\bb,\nb,\boldsymbol{p})$$ 
in the non-orientable case. 

We can equivalently replace each backbone component containing $b_i$ vertices by 
a polygon with $b_i$ sides (one of which is distinguished, corresponding to the 
first along the backbone). Thus, the numbers $\Nd_{g,k,l}(\bb,\nb)$ count the 
orientable genus $g=1+(k-b-n)/2$ connected gluings of $b$ polygons, among which 
exactly $b_i$ have $i$ sides, with $k$ pairs of sides identified in such a way 
that the boundary of the glued surface has exactly $n_i$ connected components 
consisting of $i$ sides.\footnote{For $b=1$, a similar problem was addressed in 
\cite{AS}, but the formulas derived there are considerably different from the 
ones obtained in this paper and do not stand a numerical check.}
In particular, the Harer-Zagier numbers $\epsilon_{g,k}$ \cite{HZ} that 
enumerate closed orientable genus $g$ gluings of a $2k$-gon coincide with the 
numbers $\Nd_{g,k,l}(\nb,\bb)$ with $\nb=(k-2g+1,0,0,\dots)$ and 
$\bb=\pmb{e}_{2k}$,
where we denote by $\pmb{e}_i$ the vector with 1 in the $i$-th place and $0$ 
elsewhere.

A useful notation for exponentiating a $m$-tuple $\pmb{a}=(a_1,a_2,\ldots, a_m)$ 
of variables by an integral $m$-tuple 
$\boldsymbol\alpha=(\alpha_1,\alpha_2,\ldots ,\alpha_m)$ is to write simply 
${\pmb{a}}^{\boldsymbol\alpha}=a_1^{\alpha _1}a_2^{\alpha _2}\ldots a_m^{\alpha 
_m}$;
we extend this notation in case $\pmb{a}=(a_1, a_2, a_3,\ldots )$ is a fixed 
infinite sequence of variables and $\boldsymbol\alpha$ is a finite tuple.  In 
this notation and setting $\pmb{s}=(s_0,s_1, \ldots)$ and $\pmb{t}=(t_1,t_2, 
\ldots)$, we define the orientable, multi-backbone, boundary point spectrum 
generating function $F(x,y;\pmb{s};\pmb{t})=\sum_{b\geq 1} 
F_b(x,y;\pmb{s};\pmb{t})$, where
\begin{align}
F_b(x,y;\pmb{s};\pmb{t})=
\frac{1}{b!} \sum_{k=b-1}^\infty \:\sum_{\nb}\;\sum_{\sum b_i=b} 
\Nd_{g,k,l}(\bb,\nb) x^{2g-2} y^k {\pmb{ s}}^{\nb} {\pmb t}^{\bb}\;,
\end{align}
and the non-orientable generating functions
$\tilde F(x,y;\pmb{s};\pmb{t})=\sum_{b\geq 1} \tilde F_b(x,y;\pmb{s};\pmb{t})$, 
where
\begin{align}\label{gf}
\tilde F_b(x,y;\pmb{s};\pmb{t})=
\frac{1}{b!} \sum_{k=b-1}^\infty \: \sum_{\nb}\;\sum_{\sum 
b_i=b}\widetilde{\Nd}_{g,k,l}(\bb,\nb) x^{h-2} y^k {\pmb{ s}}^{\nb} {\pmb 
t}^{\bb}\;
\end{align}
(we recall that $2-2g=b-k+\sum_{i\geq 0}n_i$ in the orientable case and 
$2-h=b-k+\sum_{i\geq 0}n_i$ in the non-orientable case, while in the both cases 
$l=\sum_{i\geq 1}in_i$).

\bigskip
\begin{thm} [Boundary point spectrum for partial chord diagrams] \label{thm1}
\noindent Consider the linear differential operators
\begin{align}
L_{0}&=\frac{1}{2} \sum_{i=0}^\infty \sum_{j=0}^i (i+2)s_j 
s_{i-j}\frac{\partial}{\partial s_{i+2}}\,,\nonumber\\
L_1&=\frac{1}{2} \sum_{i=0}^\infty(i+2)(i+1)s_i \frac{\partial}{\partial 
s_{i+2}}\,,\nonumber\\
L_2&=\frac{1}{2}\sum_{i=2}^\infty s_{i-2} \sum_{j=1}^{i-1} 
j(i-j)\,\frac{\partial^2}{\partial s_{j}\partial s_{i-j}}\,\nonumber
\end{align}
and the quadratic differential operator
\begin{align}
QF=\frac{1}{2}\sum_{i=2}^\infty s_{i-2} \sum_{j=1}^{i-1} j(i-j)\,\frac{\partial 
F}{\partial s_{j}}\cdot\frac{\partial F}{\partial s_{i-j}}\;.\nonumber
\end{align}
Then the following partial differential equations hold:
\begin{align}
&\frac{\partial F_1}{\partial y} = (L_0+x^2L_2)F_1~,\qquad\quad\;
\frac{\partial \tilde F_1}{\partial y} = (L_0+xL_1+2x^2L_2)\tilde 
F_1\,,\nonumber\\
&\frac{\partial F}{\partial y} = (L_0+x^2L_2+x^2Q)F~,\qquad
\frac{\partial \tilde F}{\partial y} = (L_0+xL_1+2x^2L_2+2x^2Q)\tilde 
F\,.\nonumber
\end{align}
These equations, together with the common for each case initial condition at 
$y=0$ given by $x^{-2}\sum_{i\geq 1} s_it_i$, determine the generating functions 
$F_1, F, \tilde F_1, \tilde F$ uniquely.
\end{thm}

Equivalently, each differential equation is solved by exponentiating $k$ times 
the operator on the right hand side applied to
$x^{-2}\sum_{i\geq 1} s_it_i$, for example, 
\begin{align*}
&F_1(x,y;s_0,s_1,\ldots 
;t_1,t_2,\ldots)=e^{y(L_0+x^2L_2)}\Big(x^{-2}\mbox{\LARGE $\Sigma$}_{i\geq 1} 
s_i t_i\Big)\;,\\
&e^{F(x,y;s_0,s_1,\ldots 
;t_1,t_2,\ldots)}=e^{y(L_0+x^2L_2)}e^{x^{-2}\mbox{\large $\Sigma$}_{i\geq 1} s_i 
t_i}\;.
\end{align*}
This explains the relationship between these differential equations and the 
corresponding enumerative problems.
These are the most efficient enumerations  of which we are aware.  As we shall 
see in the proof,
each term corresponds to adding a certain type of chord: $L_0$ and $L_2$, 
respectively, for chords with both endpoints on the same and different boundary 
components lying in a common component, $Q$ for chords whose removal separates 
the diagram, and $L_1$ the analogue of $L_0$ for M\"obius bands that give rise 
to M\"obius graphs as compared to fatgraphs in the oriented case (the subscripts 
0,1 and 2 by $L$ reflect the change in the Euler characteristic of the chord 
diagram under such an operation).  

In the last section of this paper, we provide matrix model formulas for certain 
linear combinations of the numbers 
$\Nd_{g,k,l}(\bb,\nb) $ and $\widetilde{\Nd}_{h,k,l}(\bb,\nb)$. This allows us 
to compare our computations for partial chord 
diagrams with results on a certain limiting spectral distribution, the so-called 
large $N$-limit for one backbone.
Note that a recursion for the numbers $\widetilde{\Nd}_{h,k,0}({\pmb 
e}_{2k},{\pmb n})$, for ${\pmb n}=(k-h+1,0,\dots),$ of all complete (not 
necessarily orientable) gluings
of a $2k$-gon was derived in \cite{Led} using the methods of random matrix 
theory. Our formulas specialize to those of \cite{Led}
in this particular case.

For complete chord diagrams and boundary length spectra, we shall count the 
subsets
$$\Nd_{g,k,b}(\boldsymbol{p})=\sum_{\sum 
b_i=b}\sum_{\nb}~\Nd_{g,k,0}(\bb,\nb,\boldsymbol{p})$$
in the orientable case and
$$\widetilde{\Nd}_{g,k,b}(\boldsymbol{p})=\sum _{\sum 
b_i=b}\sum_{\nb}\widetilde{\Nd}_{h,k,0}(\bb,\nb,\boldsymbol{p})$$
in the non-orientable case.
We define the orientable, multi-backbone, length spectrum generating function 
$G(x,y,t;\pmb{s})=\sum_{b\geq 1} G_{b}(x,y;\pmb{s}) t^b$, where
\begin{align}
G_{b}(x,y;\pmb{s})=
\frac{1}{b!}\sum_{k=0}^\infty \: \sum_{\boldsymbol{p}}\: 
\Nd_{g,k,b}(\boldsymbol{p}) x^{2g-2} y^k {\pmb{s}}^{\boldsymbol{p}} ,
\end{align}
and the non-orientable generating function $\tilde G(x,y,t;\pmb{s})=\sum_{b\geq 
1} \;  \tilde G_{b}(x,y;\pmb{s}) t^b$, where
\begin{align}
\tilde G_b(x,y;\pmb{s})=
\frac{1}{b!}\sum_{k\geq0} \; 
\sum_{\boldsymbol{p}}\widetilde{\Nd}_{h,k,b}(\boldsymbol{p}) x^{h-2} y^k 
{\pmb{s}}^{\boldsymbol{p}}\,.
\end{align}

\begin{thm}[Boundary length spectrum for complete chord diagrams] \label{thm2}
Define the linear differential operators
\begin{align}
K_{0}&=\frac{1}{2} \sum_{i=3}^\infty \sum_{j=1}^{i-1}
~(i-2)s_{j}s_{i-j}\frac{\partial}{\partial s_{i-2}}
~,\nonumber\\
K_1&=\frac{1}{2} \sum_{i=3}^\infty~(i-2)(i-1)\,s_{i} \frac{\partial}{\partial 
s_{i-2}}\,,\nonumber\\
K_2&=\frac{1}{2}\sum_{i=2}^\infty \sum_{j=1}^{i-1}
~j(i-j)~s_{i+2} \frac{\partial^2}{\partial s_{j}\partial s_{i-j}}\,
\end{align}
and the quadratic differential operator
\begin{align}
RG=\frac{1}{2}\sum_{i=2}^\infty s_{i+2} \sum_{j=1}^{i-1} ~j(i-j)\,\frac{\partial 
G}{\partial s_{j}}\cdot\frac{\partial G}{\partial s_{i-j}}\;.
\end{align}
Then the following partial differential equations hold:
\begin{align}
&\frac{\partial G_1}{\partial y}= (K_0+x^2K_2)G_1~,\qquad\quad
\frac{\partial \tilde G_1}{\partial y} = (K_0+xK_1+2x^2K_2)\tilde 
G_1\,,\nonumber\\
&\frac{\partial G}{\partial y}= (K_0+x^2K_2+x^2R)G~,\qquad
\frac{\partial \tilde G}{\partial y} = (K_0+xK_1+2x^2K_2+2x^2R)\tilde 
G\,.\nonumber
\end{align}
These equations, together with the common in each case initial condition at 
$y=0$ given by $x^{-2}ts_1$, determine the generating functions $G_1, G, \tilde 
G_1, \tilde G$ uniquely.
\end{thm}

\begin{rem} Complete gluing of a $2k$-gon with a marked edge can be enumerated 
in a similar way. Consider the image of the polygon perimeter, that is, the 
graph embedded in the glued surface. We say that the embedded graph has the 
\textit{vertex  spectrum} $\pmb{v}
=(v_1,v_2,\dots)$ if there are exactly $v_i$ vertices  of degree $i$.  Let 
$\widehat{N}_{g,k}(\pmb{v})$ denote the number of genus $g$ orientable gluings 
of a $2k$-gon, such that the embedded graph has the vertex spectrum $\pmb{v}$. 
The generating function
\begin{align}\label{that}
\widehat{F}(x,y;s_1,s_2,\dots) = \sum_{k=1}^\infty 
\sum_{g=0}^{[k/2]}\sum_{\pmb{v}} 
\widehat{N}_{g,k}(\pmb{v})\,x^{2g-2} y^{k-1} \pmb{s}^{\pmb{v}}
\end{align}
for the numbers $\widehat{N}_{g,k}(\pmb{v})$ satisfies the equation
\begin{align}\label{fhatpde}
\frac{\partial\widehat{F}(x,y;\pmb{s})}{\partial y}=(K_0+x^2 
K_2)\widehat{F}(x,y;\pmb{s}),
\end{align}
and is uniquely determined by it together with the initial condition
$$\widehat{F}\left|_{y=0}\right.=x^{-2}s_1^2.$$
Actually, $\widehat{F}$ and $G_1$ are explicitly related by the formula
\begin{align}
\frac12\Lambda_1\widehat{F}=
\frac{\partial G_1}{\partial y}\,,\qquad \Lambda_1=\sum_{i=1}^\infty 
is_{i+1}\frac{\partial}{\partial s_i}\;
\end{align}
(this immediately follows from the fact that both $K_0$ and $K_2$ commute with $\Lambda_1$).
The same problem, but differently formulated (namely, the enumeration of genus 
$g$ fatgraphs on $n$ vertices of specified degrees) was recently solved in 
\cite{MS}. However, our generating function (\ref{that}) for these numbers and 
the partial differential equation  (\ref{fhatpde}) it satisfies are different 
from their counterparts in \cite{MS}.
\end{rem}

The following observation we learned from M.~Kazarian \cite{Ka1,Ka2}: for $x=1$ 
the generating functions $F\left|_{x=1}\right.$ and $G\left|_{x=1}\right.$
satisfy an infinite system of non-linear partial differential equations called 
the KP (Kadomtsev-Petviashvili) hierarchy (in particular, this means that the 
numbers $\Nd_{g,k,l}(\bb,\nb)$ and ${\Nd}_{g,k,b}(\boldsymbol{p})$ additionally 
obey an infinite system of recursions). The KP hierarchy is one of the best 
studied completely integrable systems in mathematical physics. Below are the 
several first equations of the hierarchy:
\begin{align}\label{KP}
&F_{22}=-\frac12\,F_{11}^2+F_{31}-\frac1{12}\,F_{1111}\;,\nonumber\\
&F_{32}=-F_{11}F_{21}+F_{41}-\frac16F_{2111}\;,\nonumber\\
&F_{42}=-\frac12\,F_{21}^2-F_{11}F_{31}+F_{51}+\frac18\,F_{111}^2
+\frac1{12}\,F_{11}F_{1111}-\frac14\,F_{3111}+\frac1{120}\,F_{111111}\;,
\nonumber\\
&F_{33}=\frac13\,F_{11}^3-F_{21}^2-F_{11}F_{31}+F_{51}
+\frac14\,F_{111}^2+\frac13\,F_{11}F_{1111}-\frac13\,F_{3111}+\frac1{45}\,F_{
111111}\;,
\end{align}
where the subscript $i$ stands for the partial derivative with respect to $s_i$. 
The exponential $e^F$ of any solution is called a {\em tau function} of the 
hierarchy. The space of solutions (or the space of tau functions) has a nice 
geometric interpretation as an infinite-dimensional Grassmannian (called the 
{\em Sato Grassmannian}), see, e.~g., \cite{MJD} or \cite{Ka1} for details.
See also \cite{AU1} for another application of
the Sato Grassmannian to conformal field theory.
The 
space of solutions is homogeneous: there is a Lie algebra 
$\widehat{\mathfrak{gl}(\infty)}$ (a central extension of 
$\mathfrak{gl}(\infty)$) that acts infinitesimally on the space of solutions, 
and the action of the corresponding Lie group is transitive.

Introduce the standard bosonic creation-annihilation operators
\begin{align*}
a_i=\begin{cases} s_i &{\rm if}\; i>0\\ 0 & {\rm if}\; i=0\\ 
(-i)\frac{\partial}{\partial s_{-i}} & {\rm if}\; i<0 \end{cases}
\end{align*} 
and put
\begin{align*}
&\varLambda_m=\frac{1}{2}\sum_{i=-\infty}^{\infty}a_i\,a_{m-i}\;,\\
&M_m=\frac{1}{6}\sum_{i,j=-\infty}^{\infty}:\!a_i\,a_j\,a_{m-i-j}\!:
\end{align*}
(the notation $:\!a_{i_1}\dots a_{i_r}\!\!:$ stands for the ordered product 
$a_{i_{\sigma(1)}}\dots a_{i_{\sigma(r)}}$, where $\sigma\in S_r$ is a 
permutation such that $i_{\sigma(1)}\geq\dots\geq i_{\sigma(r)}$).\footnote{The 
operator $M_0$ is the famous cut-and-join operator \cite{GJ} used in the 
computation of Hurwitz numbers.}
All the operators $a_i,\; \varLambda_m, \; M_m$ belong to the Lie algebra 
$\widehat{\mathfrak{gl}(\infty)}$. Moreover, it is easy to check that 
\begin{align}
L_0+L_2=s_0^2\frac{\partial}{\partial s_2} + s_0\varLambda_{-2} + 
M_{-2}\,,\qquad K_0+K_2=M_2\,,
\end{align}
so that $L_0+L_2$ and $K_0+K_2$ also belong to 
$\widehat{\mathfrak{gl}(\infty)}$. Now we notice that the exponentials 
$e^{\sum s_it_i}$ and $e^{ts_1}$ of the initial conditions in Theorems 1 and 2 
both are KP tau functions for a trivial reason -- their logarithms are linear in 
$s_1,s_2,\dots$ and therefore obviously satisfy the equations of KP hierarchy 
(\ref{KP}) for any values of the other parameters. Moreover, both $e^{L_0+L_2}$ and 
$e^{K_0+K_2}$ preserve the Sato Grassmannian and map KP tau functions to KP tau 
functions. Thus, $e^{F\left|_{x=1}\right.}=e^{L_0+L_2}e^{\sum_{i\geq 1} s_it_i}$ 
and $e^{G\left|_{x=1}\right.}=e^{K_0+K_2}e^{ts_1}$ are KP tau function as well, 
and we get

\begin{cor}[M.~Kazarian \cite{Ka2}]\label{tau}
The generating functions 
$$F\left|_{x=1}\right.=F(1,y;s_0,s_1,\dots;t_1,t_2,\dots)\;{\it and}\; 
G\left|_{x=1}\right.=G(1,y,t;s_1,s_2,\dots)$$
satisfy the infinite system of KP equations (\ref{KP}) with respect to 
$s_1,s_2,\dots$ for any values of the parameters $y,t,t_1,t_2,\dots$. 
Equivalently, the partition functions $e^{F\left|_{x=1}\right.}$ and 
$e^{G\left|_{x=1}\right.}$ are (multi-parameter) families of KP tau functions.
\end{cor}

Let us now comment on the relevance of the above results to describing the RNA 
interactions. 
Define $C_{g,b,k}=\sum_{\boldsymbol{p}} {\mathcal N}_{g,k,b}(\boldsymbol{p})$ to 
be the number of complete and connected chord diagrams of genus $g$ on $b$ 
ordered and oriented backbones with $k$ chords, so in particular,
$C_{g,1,k}$ is the Harer-Zagier number $\epsilon_{g,k}$.
These chord diagrams provide the basic model for a complex of interacting RNA 
molecules, one RNA molecule for each backbone and one chord for
each Watson-Crick\footnote
{These are the allowed bonds G-C and A-U between nucleic acids.  For the expert, 
let us emphasize that any other model including wobble G-U or further exotic
base pairs is handled in exactly the same way with one chord for each allowed 
type of bond.
} bond between nucleic acids, where one demands that the chord endpoints respect 
the 
natural ordering\footnote{From the so-called 5' to 3' end as determined by the 
chemical structure of the RNA.} of the nucleic acids in each molecule, i.e, in 
each oriented backbone.  It is very natural, as is the attention to connected 
chord diagrams in order to avoid separate molecular interactions.   In reality, 
RNA folds according to a {\sl partial} chord diagram, i.e., there are in 
practice unbonded nucleic acids.\footnote {Typically, 50 to 80 percent of 
nucleic acids participate in Watson-Crick base pairs together with several 
percent exotic.  On the other hand in an extreme example, roughly 50 percent are 
Watson-Crick and 40 percent exotic for  ribosomal RNA.}

Recall from \cite{APRW} that a shape
is a special connected and complete chord diagram which has no parallel chords, 
has a unique ``rainbow" on each backbone, i.e., a chord whose endpoints are 
closer to the backbone endpoints than any other chord and
no ``1-chords'' connecting vertices consecutive in a single backbone unless the 
1-chord is a rainbow.
In the very special (genus zero on one backbone) case, the single-chord diagram 
is permitted since
the 1-chord is a rainbow, but in all other cases, there are no 1-chords, each 
backbone has a unique rainbow,
and ${p}_1=0$, ${p}_2=b$. If a shape is not the special single-chord diagram and 
we remove its $b$ rainbows, then
the resulting diagram has ${p}_1=0={p}_2$.  Conversely, in a chord diagram with 
${p}_1=0={p}_2$, no backbone has a rainbow, and $b$ rainbows can be added to 
produce a shape.
Let $S_{g,b,k}$ denote the number of shapes of genus $g$ on $b$ backbones with 
$k$ chords.

Define the generating functions $C(x,y,t) =  \sum_{g=0}^\infty  
\sum_{b=1}^\infty  C_{g,b}(y)x^{2g-2} t^b$, with
$$C_{g,b}(y) = \frac{1}{b!}\sum_{k=2g+b-1}^\infty C_{g,b,k}  \, y^k ,$$
and $S(x,y,t) = \sum_{g=0}^\infty \sum_{b=1}^\infty S_{g,b}(y) x^{2g-2}t^b $, 
with
$$
S_{g,b}(y) =\frac{1}{b!}\sum_{k=2g+2b-1}^{6g-6+5b} S_{g,b,k} y^k.
$$
It follows by construction that 
$$C(x,y,t) = F(x,y;1,0,0,\ldots; t, t, \ldots)=G(x,y,t;1,1,\cdots )$$ and 
$S(x,y,t) = 1 + G(x,y,t;0,0,1,1,\ldots)$, so we have computed
here both the complete chord diagrams $C_{g,b,k}$ and the shapes 
$S_{g,b,k}$\footnote{Furthermore, the free energy $F_g(s,t)$ for the matrix 
model in \cite{ACPRS}  is given (up to a constant depending only on $g$ times 
$N^{2-2g}$)
by our $G(N^{-1}, t^\frac{1}{2},s;1,1,1...)$.}.
In fact \cite{ACPRS},
the generating functions for shapes and chord diagrams are algebraically related 
by
$$\aligned
C_{g,b}(z)&=\left(\frac{1}{z C_0(z)}\right)^b~ 
S_{g,b}\left(\frac{C_0(z)-1}{2-C_0(z)}\right),\cr
S_{g,b}(z)&=\left(\frac{z}{1+2z}\right)^b~ 
C_{g,b}\left(\frac{z(1+z)}{(1+2z)^2}\right),\cr
\endaligned$$
where $C_0(z)=\frac{1-\sqrt{1-4z}}{2z}$ is the Catalan generating function, the 
former equation expressing the formal power series
$C_{g,b}(z)$ in terms of the polynomial $S_{g,b}(z)$.
As a further interesting open problem, inspired by the results of this paper, we ask
if there is a non-zero finite order differential operator in the variables $(x, y, t)$
which together with an initial condition determines $C(x, y, t)$?

One point about shapes is that standard combinatorial techniques allow their 
``inflation'' to 
complete chord diagrams as indicated in the previous formulas, and furthermore, 
complete chord diagrams can likewise be inflated to
partial chord diagrams, cf.\ \cite{APRW,Reidysbook}.  Another point is that 
shape inflation is well-suited to the accepted Ansatz for free energy
and so provides efficient polynomial-time algorithms for computing minimum free 
energy RNA folds \cite{gfold,Reidysbook}
at least in the planar case.
A further geometric point \cite{APRW} is that shapes of genus $g$ on $b$ 
backbones are dual to cells in the Harer-Mumford-Strebel \cite{Strebelbook} or 
Penner \cite{Penner87}
decomposition of Riemann's moduli space of genus $g$ surfaces with $b$ boundary 
components provided $2g-2+b>0$.

As was already discussed, it is really partial chord diagrams that actually 
describe complexes of RNA molecules with its
distillation first to complete chord diagrams and then to shapes.  All three 
formulations of the combinatorics have thus been
treated here, namely, shapes and complete chord diagrams by the previous 
formulas and partial chord diagrams by inflation
or instead directly with our generating function in Theorem \ref{thm1}.

This paper is organized as follows.  Section 2 contains basic combinatorial 
results on the boundary point spectra of chord diagrams on one backbone and 
derives the equation given before on $F_1$ (Proposition \ref{1b}), and section 3 
extends these results to include possibly separating edges and derives the 
equation given before on $F$ (Proposition \ref{mb}).  Boundary point spectra of 
non-orientable surfaces are discussed in Section 4, and the equations given 
before on $\tilde F_1$ and $\tilde F$ are derived 
(Proposition \ref{nonor_thm}), so together
Propositions \ref{1b}-\ref{nonor_thm} comprise Theorem 1.
Section 5 is dedicated to boundary length spectra, and the situation is similar 
to boundary point spectra in that each counts
data for each fatgraph boundary cycle.  For this reason, the arguments are only 
sketched for boundary length spectra
culminating in the equations from before on $G_1$, $G$, $\tilde G_1$ and $\tilde 
G$ (Theorem 2). Section 6 introduces
a random matrix technique for partial chord diagrams and provides a matrix 
integral for boundary point spectra computations in both the orientable and 
non-orientable cases. Free probability techniques permit the computation of the 
large-N limit which reproduces computations based on the partial differential 
equations, providing a consistency check on the entire discussion. 

\section{Combinatorics of connected partial chord diagrams}\label{secct}
As before, $\pmb{e}_i$ denotes the sequence $(0,\dots,0,1,0,\dots)$ with 1 in 
the $i$-th place and 0 elsewhere. 
We say simply that a diagram is of type $\{ g,k,l;\bb;\nb\}$ if it is of type 
$\{ g,k,l;\bb,\nb,\boldsymbol{p}\}$ 
for some $\boldsymbol{p}$ and let
$\Nd_{g,k,l}(\bb,\nb) = 0$ if there are no diagrams of type $\{ 
g,k,l;\bb;\nb\}$.

\begin{prop} \label{rec_forv2}
The numbers $\Nd_{g,k,l}(\pmb{e}_{2k+l},\nb)$ enumerating one backbone chord 
diagrams of type $\{g,k,l;\pmb{e}_{2k+l};\nb\}$ obey the following recursion 
relation:
\begin{align}
k\,\Nd_{g,k,l}(\pmb{e}_{2k+l},&\,\nb)=\nonumber\\
\frac12 &\sum_{i=0}^\infty \sum_{j=0}^{i} (i+2)(n_{i+2}+1)
\Nd_{g,k-1,l+2}(\pmb{e}_{2k+l}, 
\nb-\pmb{e}_j-\pmb{e}_{i-j}+\pmb{e}_{i+2})+\nonumber\\
\frac12 &\sum_{i=0}^\infty \sum_{j=1}^{i+1} j(i+2-j)
(n_j+1+\delta_{j,i+2-j}-\delta_{i,j})(n_{i+2-j}+1-\delta_{j,2})\times\nonumber\\
&\hspace{1in}\Nd_{g-1,k-1,l+2}(\pmb{e}_{2k+l}, 
\nb+\pmb{e}_j+\pmb{e}_{i+2-j}-\pmb{e}_{i}).
\label{big_sumv2}
\end{align}
\label{rec_for}
\end{prop}

\begin{proof}
Let us start with a chord diagram of type $\{g,k,l;\pmb{e}_{2k+l};\nb\}$. Note 
that erasing a chord in a diagram, we keep its endpoints as marked points. This 
yields two possibilities. 

The first possibility is that the chord belongs to two distinct boundary 
components, say, one with $j$ and the other with $i-j$ marked points. After 
erasing the chord, these two boundary components join into one component with 
$i+2$ marked points, and the genus of the diagram does not change (see Fig. 
\ref{cd_b}). Thus, one gets a diagram  of genus $g$ with $k-1$ chords, $l+2$ 
marked points and boundary point spectrum 
$\nb-\pmb{e}_j-\pmb{e}_{i-j}+\pmb{e}_{i+2}$. 

\begin{figure}[hbt]%
 \begin{center}
 \includegraphics[width=7cm]{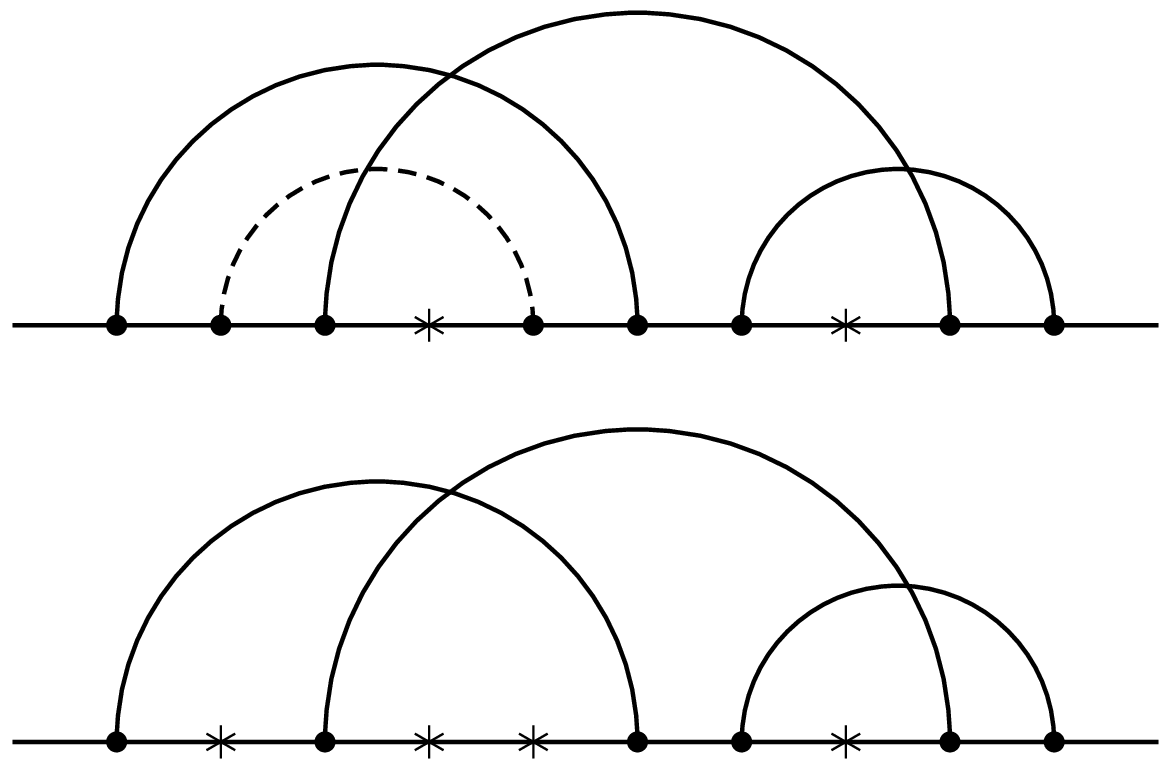}
 \caption{Erasing the dashed chord changes the diagram type  from 
$\{1,4,2;\pmb{e}_{10};2\pmb{e}_0+\pmb{e}_2\}$ to 
$\{1,3,4;\pmb{e}_{10};\pmb{e}_0+\pmb{e}_4\}$}
 \label{cd_b}%
\end{center}
\end{figure}

\begin{figure}[hbt]%
 \begin{center}
 \includegraphics[width=7cm]{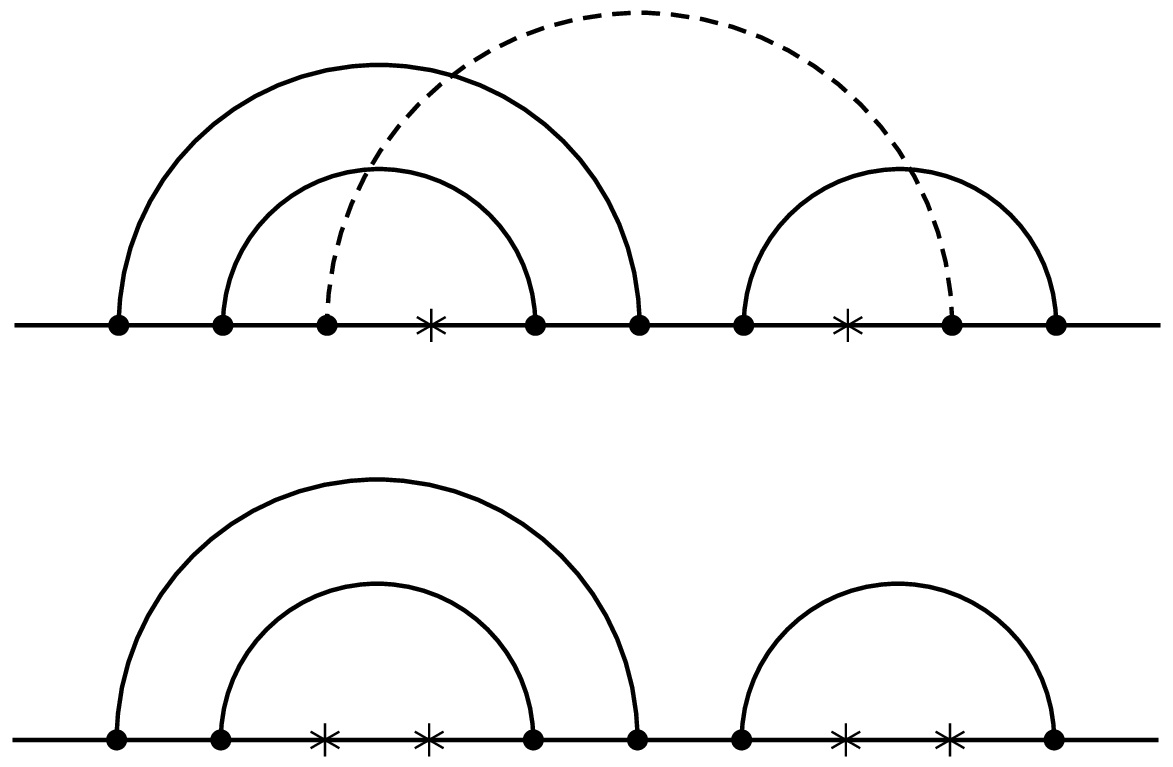}
 \caption{Erasing the dashed chord changes the diagram type  from 
$\{1,4,2;\pmb{e}_{10};2\pmb{e}_0+\pmb{e}_2\}$ to 
$\{0,3,4;\pmb{e}_{10};2\pmb{e}_0+2\pmb{e}_2\}$}
 \label{cd_a}%
\end{center}
\end{figure}

The second possibility is that one boundary component traverses the chord twice, 
i.e., once in each direction. Erasing this chord splits the boundary component 
(say, with $i$ marked points) into two (with $j$ and $i-j+2$ marked points 
respectively, $0\leq j\leq i+1$) (see Fig. \ref{cd_a}). In this case, one gets a 
chord diagram of genus $g-1$ with $k-1$ chords, $l+2$ marked points and boundary 
point spectrum $\nb+\pmb{e}_j+\pmb{e}_{i+2-j}-\pmb{e}_{i}$.

In order to prove (\ref{big_sumv2}), let us compute the number of chord diagrams 
of type $\{g,k,l;\pmb{e}_{2k+l};\nb\}$ with one marked chord in two different 
ways. On the one hand, there are $k$ possibilities to mark a chord in a diagram 
with $k$ chords, so the number in question is $k\, \Nd_{g,k,l}(\pmb{e}_{2k+l}, 
\nb)$. On the other hand, one can join any two marked points with a marked chord 
on any diagram with $k-1$ chords. We have described above all types of diagrams 
with $k-1$ chords that could potentially give a $k$-chord diagram of the 
required type after adding a chord.  

If one takes a diagram of type 
$\{g,k-1,l+2;\pmb{e}_{2k+l};\nb-\pmb{e}_j-\pmb{e}_{i-j}+\pmb{e}_{i+2}\}$ (let us 
assume that $j<i-j$), then there are $n_{i+2}+1$ possibilities to choose a 
boundary component with $i+2$ marked points. One then needs to connect two 
marked points on it with a chord in such a way that it splits into two boundary 
components with $j$ and $i-j$ marked points respectively. This can be done in 
$(i+2)$ different ways. If $j=i-j$, then there are $\frac{i+2}{2}=j+1$ ways to 
split the boundary component into two components with $j$ marked points each. 
For $j>i-j$ we get the same diagrams as in the case $j<i-j$, hence we get the 
first term on the r.h.s. of (\ref{big_sumv2}).

If one takes a diagram of type 
$\{g-1,k-1,l+2;\pmb{e}_{2k+l};\nb+\pmb{e}_j+\pmb{e}_{i+2-j}-\pmb{e}_{i}\}$ (let 
us assume that $j<i-j+2$), then there are $(n_{j}+1)$ ways to choose a boundary 
component with $j$ marked points, provided $j\neq i$. If $j=i$, then $j<2=i+2-j$ 
and so $j=1=i$ and the number of ways is then $n_1$. There are $(n_{i-j+2}+1)$ 
ways to choose a boundary component with $i-j+2$  marked points if $i\neq 
i+2-j$. If $i+2-j = i$, then the number of ways is $n_i$. One then needs to 
connect with a chord a marked point on one  boundary component with a marked 
point on the other one. This can be done in $j(i-j+2)$ different ways. If 
$j=i-j+2$, then there are $(n_{j}+2)(n_j+1)/2$ ways to choose a pair of boundary 
components with $j$ marked points, provided $i\neq j$. If we have $i=j$ and also 
$j=i-j+2$, then $i=2=j$ and the number of ways is $(n_2+1)n_2/2$. In both cases, 
there are $j^2$ ways to connect with a chord two points on different components. 
This gives us the second term on the r.h.s. of (\ref{big_sumv2}).         
\end{proof}

\begin{prop}\label{1b}
The one backbone generating function $F_1(x,y;s_0,s_1,\dots;t_1,t_2,\dots)$ is 
uniquely determined by the equation 
\begin{align}
\frac{\partial F_1}{\partial y}=LF_1\,,\quad L=L_0+x^2L_2\,, \label{pde}
\end{align}
together with the initial condition
\begin{align}
F_1(x,0;s_0,s_1,\dots;t_1,t_2,\dots)=\frac{1}{x^2}\sum_{i=1}^\infty s_i t_i\;. 
\label{ic}
\end{align}
Equivalently, we have
\begin{align}
F_1(x,y;s_0,s_1,\dots;t_1,t_2,\dots)=e^{yL}\left(\frac{1}{x^2}\sum_{i=1}^\infty 
s_i t_i\right)\;.
\end{align}
\end{prop}

\begin{proof} It is straightforward to check that the equation $\frac{\partial 
F_1}{\partial y}=LF_1$ is equivalent to formula 
(\ref{big_sumv2}). Moreover, every chord diagram of type 
$\{g,k,l;\pmb{e}_{2k+l};\nb\}$ can be obtained from the unique diagram of type 
$\{0,0,2k+l;\pmb{e}_{2k+l}; \pmb{e}_{2k+l}\}$ by adding $k$ chords to it. On the 
level of $F_1$, this amounts to applying the operator $L$ to 
$x^{-2}s_{2k+l}t_{2k+l}$ precisely $k$ times and taking the coefficient of the 
monomial $x^{2g-2}t_{2k+l}s_0^{n_0}s_1^{n_1}\dots$ in 
$L^k(x^{-2}s_{2k+l}t_{2k+l})$ which is equal to 
$k!\,\Nd_{g,k,l}(\pmb{e}_{2k+l},\nb)$ by formula (\ref{big_sumv2}).
\end{proof}

\begin{rem}\label{rem_coeff} Proposition \ref{1b} allows us to compute the 
numbers $\Nd_{g,k,l}(\pmb{e}_{2k+l},\nb)$ reasonably quickly. For instance, take 
$g=0$ and put $s_0=1,\;s_i=qy^is^i,\;i\geq 1,\;t_j=1,\;j\geq 1$. The several 
first coefficients of 
$x^2F_1(x,y^2;1,qys,q(ys)^2,\dots;1,1,\dots)$ in $y$ for $x=0$ and $k=0,\dots,8$ 
are:
\begin{align*}
&k=0:\quad 1\\
&k=1:\quad q s\\
&k=2:\quad q s^2 + 1 \\
&k=3:\quad q s^3  + 3 q s\\
&k=4:\quad q s^4  + (4 q + 2 q^2) s^2+2   \\
&k=5:\quad q s^5  + (5 q + 5 q^2) s^3  + 10 q s\\
&k=6:\quad q s^6  + (6 q + 9 q^2) s^4  + (15 q + 15 q^2) s^2+5\\ 
&k=7:\quad q s^7  + (7 q + 14 q^2) s^5  + (21 q + 42 q^2  + 7 q^3)s^3  + 35 q 
s\\
&k=8:\quad q s^8  + (8 q + 20 q^2)s^6  + (28 q + 84 q^2  + 28 q^3)s^4 + (56 q + 
84 q^2)s^2+ 14 
\end{align*}
In Section \ref{secmi}, we will derive these same polynomials by matrix 
integration methods.
\end{rem}

\section{The multibackbone case}
Let us proceed with the multibackbone case. 

\begin{prop} \label{qrec}
The numbers $\Nd_{g,k,l}(\bb,\nb)$ obey the following recursion relation:
\begin{align}
k\, \Nd_{g,k,l}&(\bb, \nb)=\nonumber\\
\frac{1}{2}&\sum_{i=0}^\infty \sum_{j=0}^i (i+2)(n_{i+2}+1)\, 
\Nd_{g,k-1,l+2}(\bb, \nb-\pmb{e}_j-\pmb{e}_{i-j}+\pmb{e}_{i+2})\;+\nonumber\\
\frac{1}{2}&\sum_{i=0}^\infty \sum_{j=1}^{i+1} 
j(i+2-j)(n_j+1+\delta_{j,i+2-j}-\delta_{i,j})(n_{i+2-j}+1-\delta_{j,2})\times 
\nonumber\\
&\hspace{0.4in}\Nd_{g-1,k-1,l+2}(\bb, 
\nb+\pmb{e}_j+\pmb{e}_{i+2-j}-\pmb{e}_{i})+\nonumber\\
\frac{1}{2}&\sum_{i=0}^\infty\; \sum_{j=1}^{i+1}\; \sum_{g_1+g_2=g} \;\; 
\sum_{k_1+k_2=k-1}\;\;   
\sum_{\nb^{(1)}+\nb^{(2)}=\nb-\eb_{i}} \;\; 
\sum_{\bb^{(1)}+\bb^{(2)}=\bb}\nonumber\\
&\hspace{0.3in}j(i+2-j)(n^{(1)}_j+1)(n^{(2)}_{i+2-j}+1)\, 
\frac{b!}{b^{(1)}!b^{(2)}!}\times\nonumber\\ 
&\hspace{0.4in}\Nd_{g_1,k_1,l_1+j}(\bb^{(1)}, 
\nb^{(1)}+\eb_{j})\,\Nd_{g_2,k_2,l_2+i+2-j}(\bb^{(2)}, \nb^{(2)}+\eb_{i+2-j}),
\label{multi_comb}
\end{align}
where 
$$b^{(r)}=\sum_{i=1}^\infty b_i^{(r)},\quad l_r= \sum_{i=1}^\infty 
in_i^{(r)},\quad\sum_{i=0}^\infty n_i^{(r)}=k_r-2g_r-b^{(r)}+2\,,\quad 
r=1,2\,.$$
\end{prop}

\begin{proof}
The multibackbone case is similar to the one backbone case, and the derivation 
of the first two sums on the r.h.s. of (\ref{multi_comb}) repeats verbatim the 
proof of (\ref{big_sumv2}), cf. Proposition \ref{rec_for}.  The only difference 
is that erasing a chord can split the diagram into two connected components (see 
Fig. \ref{cd_c}).

\begin{figure}[hbt]%
 \begin{center}
 \includegraphics[width=8cm]{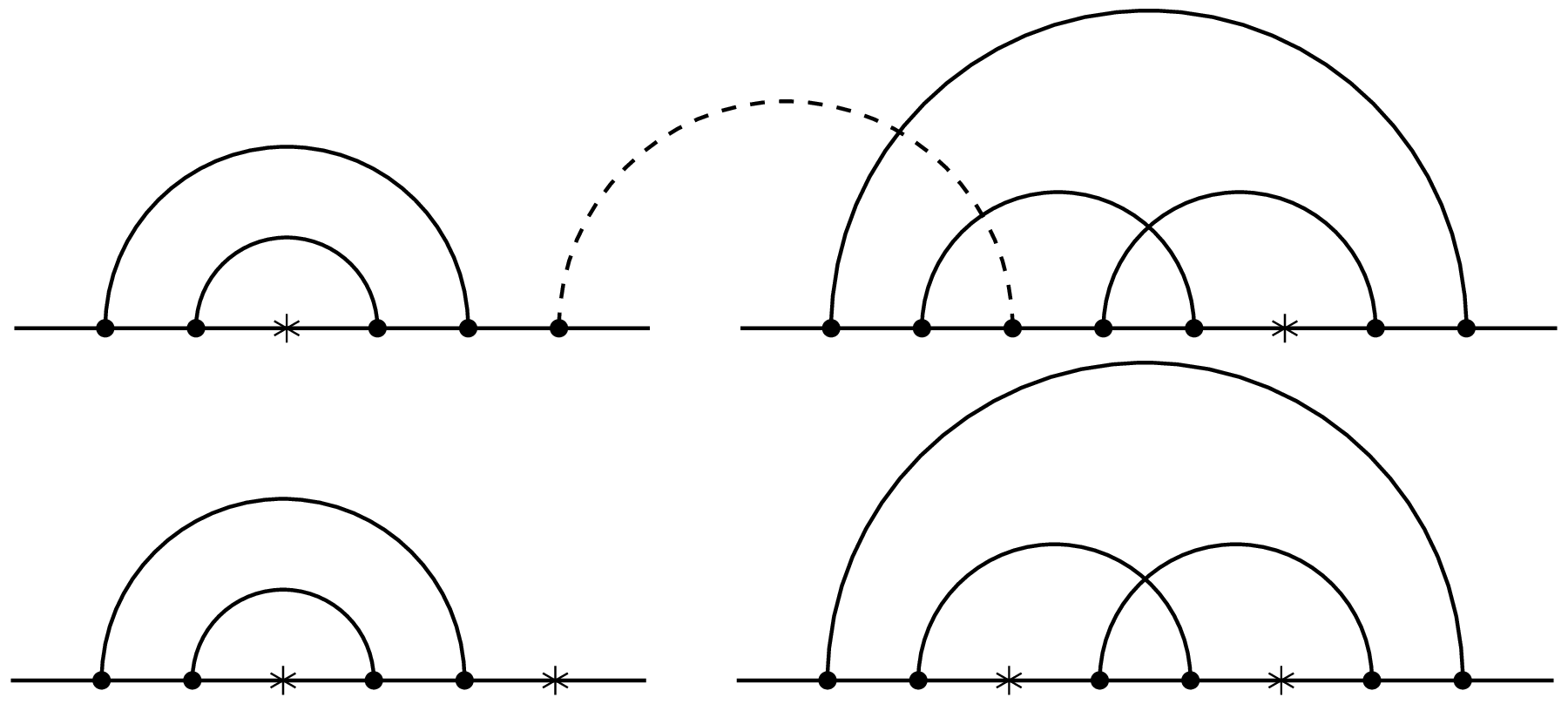}
 \caption{Erasing the dashed chord splits the chord diagram of type 
$\{1,6,2;\pmb{e}_6+\pmb{e}_8; 2\pmb{e}_0+2\pmb{e}_1\}$ into two diagrams of 
types 
 $\{0,2,2;\pmb{e}_{6}; \pmb{e}_0+2\pmb{e}_1\}$ and $\{1,3,2;\pmb{e}_{8}; 
\pmb{e}_0+\pmb{e}_2\}$}
 \label{cd_c}%
\end{center}
\end{figure}

 This possibility is encoded in the 6-fold sum on the r.h.s. of 
(\ref{multi_comb}). There are exactly 
$$j(i+2-j)n^{(1)}_jn^{(2)}_{i+2-j} \frac{b!}{b^{(1)}!b^{(2)}!}$$ 
ways to get a chord diagram of a type $\{g,k,l;\bb;\nb\}$ from two diagrams of 
types $\{g_1,k_1,l_1;\bb^{(1)}; \nb^{(1)}\}$ and $\{g_2,k_2,l_2;\bb^{(2)}; 
\nb^{(2)}\}$. Namely, there are $n^{(1)}_j$ ways to choose a boundary component 
with $j$ marked points on the first diagram, and there are $n^{(2)}_{i+2-j}$ 
ways to choose a boundary component with $i-j+2$ marked points on the second 
diagram. There are $j(i-j+2)$ ways to connect a marked point on the first 
boundary component with a marked point on the second one. The remaining factor 
$\frac{b!}{b^{(1)}!b^{(2)}!}$ counts the number of different ordered splittings 
of a $b$-backbone diagram into two connected ones that contain $b^{(1)}$ and  
$b^{(2)}$ backbones respectively.
\end{proof}

\begin{prop}\label{mb}
The generating function $F(x,y;s_0,s_1,\dots;t_1,t_2,\dots)$ is uniquely 
determined by the equation 
\begin{align}
\frac{\partial F}{\partial y}=(L+x^2Q)F, \label{pde2}
\end{align}
where $L=L_0+x^2L_2$ and 
\begin{align}
QF=\frac{1}{2}\sum_{i=2}^\infty s_{i-2} \sum_{j=1}^{i-1} j(i-j)\,\frac{\partial 
F}{\partial s_{j}}\cdot\frac{\partial F}{\partial s_{i-j}}\;,
\end{align}
together with the same initial condition
\begin{align}
F(x,0;s_0,s_1,\dots;t_1,t_2,\dots)=\frac{1}{x^2}\sum_{i=1}^\infty s_i t_i\;. 
\label{ic_multi}
\end{align}
Equivalently, the multibackbone partition function
$e^F$ satisfies the equation 
$$\frac{\partial e^F}{\partial y}=Le^F$$ and is explicitly given by
\begin{align}
e^{F(x,y;s_0,s_1,\dots;t_1,t_2,\dots)}=e^{yL}\left(e^{\frac{1}{x^2}\sum_{i=1}
^\infty s_i t_i}\right)\;.
\end{align}
\end{prop}

\begin{proof}  As in the one backbone case, a straightforward computation shows 
that recursion (\ref{multi_comb}) is equivalent to the equation $\frac{\partial 
F}{\partial y}=(L+x^2Q)F$ (where the 6-fold sum translates into the quadratic 
term $QF$). Moreover, every chord diagram of type $\{g,k,l;\nb;\bb\}$ can be 
obtained from the disjoint collection of $b$ diagrams of type 
$\{0,0,i;\pmb{e}_{i},\pmb{e}_{i}\}$ (each taken with multiplicity $b_i$) by 
connecting them with $k$ chords. Let $F^{(k)}(x;s_0,s_1,\dots;t_1,t_2, \dots)$ 
be the coefficient of $y^k$ in the total generating function $F$, so 
$$F^{(k+1)}(x;s_0,s_1,\dots;t_1,t_2, \dots)$$
is the the coefficient of $y^{k+1}$ in $$y(L+x^2Q)\left(\sum_{i=0}^{k} 
y^iF^{(i)}(x;s_0,s_1,\dots;t_1,t_2, \dots)\right),$$
where $F^{(0)}(x;s_0,s_1,\dots;t_1,t_2, \dots)=\frac{1}{x^2}\sum_{i\geq 
1}s_it_i$.
\end{proof}

\begin{rem}
The enumeration problem of {\sl complete} (i.e., giving a closed surface) 
orientable gluings of two and three polygons (or equivalently chord diagrams on 
$2$ or $3$ backbones without marked points) was solved in \cite{APRW} and 
independently in \cite{PR}by different methods.  
\end{rem}

\section{Non-orientable polygon gluings}\label{sec:nonor}

This section is dedicated to proving the following result:

\begin{prop}
The one backbone generating function $\tilde 
F_1(x,y;s_0,s_1,\dots;t_1,t_2,\dots)$ is uniquely determined by the equation 
\begin{align}
\frac{\partial \tilde F_1}{\partial y}=(L_0 + x L_1 + 2x^2 L_2) \tilde F_1 
\label{pde_non}
\end{align}
together with the initial condition
\begin{align}
\tilde F_1(x,0;s_0,s_1,\dots;t_1,t_2,\dots)=\frac{1}{x^2}\sum_{i=1}^\infty s_i 
t_i\;. \label{ic_non}
\end{align}
The generating function $\tilde F(x,y;s_0,s_1,\dots;t_1,t_2,\dots)$ is uniquely 
determined by the equation 
\begin{align}
\frac{\partial \tilde F}{\partial y}=(L_0 + x L_1 + 2x^2 L_2 + 2x^2Q) \tilde F 
\;,\label{pde2_non}
\end{align}
together with the same initial condition (\ref{ic_non}).
\label{nonor_thm}
\end{prop}

\begin{proof}
The non-orientable case is similar to the orientable one. On the combinatorial 
level, the difference is that when one glues two sides on the same  connected 
component of a boundary with a twist, one adds a cross-cap to the surface 
without changing the number of boundary components. On the level of the 
generating function $\tilde F_1$, this adds the term $x L_1\tilde F_1$ on the 
r.h.s. of (\ref{pde_non}). If one glues two sides belonging to distinct 
components of the boundary, then there is no difference between the twisted and 
untwisted gluings, so that one just has to count the term $x^2L_2\tilde F_1$ 
twice. The multibackbone generating function is treated analogously.
\end{proof}

Using Proposition \ref{nonor_thm}, we can compute several first numbers 
$\widetilde{\Nd}_{h,k,l}(\bb,\pmb{e}_{2k+l})$. Consider, for example, the 
decagon gluings, i.e., $2k+l=10$. For $x=1$, the coefficients of the generating 
series $\tilde F_1(1,y;\pmb{e}_{10};s_0,s_1,\dots)$ in $y$ are listed below for 
$k=0,1,\dots,5$:
\begin{align*}
k=0:\quad &s_{10}\;,\\
k=1:\quad &10s_0s_8+10s_1s_7+10s_2s_6+10s_3s_5+5s_4^2+45s_8\;,\\
k=2:\quad &45s_0^2s_6+90s_4s_0s_2+90s_3s_1s_2+325s_0s_6+300s_1s_5+285s_2s_4\\
&+1050s_6+45s_4s_1^2+45s_0s_3^2+140s_3^2+15s_2^3+90s_0s_1s_5\;, \\
k=3:\quad 
&1850s_0s_1s_3+360s_0^2s_1s_3+1000s_0^2s_4+360s_0s_1^2s_2+900s_0s_2^2\\
&+870s_1^2s_2+4900s_4s_0+4100s_3s_1+120s_0^3s_4+30s_1^4\\
&+180s_0^2s_2^2+1920s_2^2+8610s_4\;,\\
k=4:\quad 
&1720\,s_{0}^3s_{2}+2465\,s_{{0}}^{2}s_{{1}}^{2}+8890\,s_{0}^{2}s_{{2}}+7940\,s_
{{0}}s_{{1}}^{2}+21930\,s_{{0}}s_{{2}}\\
&+420\,s_{{0}}^{3}s_{{1}}^{2}+210\,s_{{0}}^{4}s_{{2}}+9120\,s_{{1}}^{2}+22905\,
s_{{2}}\;,\\
k=5:\quad &42s_0^6+386s_0^5+2290s_0^4+7150s_0^3+12143s_0^2+8229s_0\;.
\end{align*}

\section{Enumeration of chord diagrams with fixed boundary lengths}

We will prove Theorem \ref{thm2} in analogy to Theorem \ref{thm1} by 
combinatorial methods. The partial differential equation on $G$ is equivalent to 
the following
\begin{prop}
The numbers $\Nd_{g,k,b}(\boldsymbol{p})$  obey the following recursion 
relation:
\begin{align}
k\,&\Nd_{g,k,b}(\boldsymbol{p}) = \nonumber\\
&\frac{1}{2} \sum_{i=1}^\infty\sum_{j=0}^i i 
(p_{i}+1-\delta_{i,j+1}-\delta_{j,1})\,\Nd_{g,k-1,b}(\boldsymbol{p}+\eb_{i}-\eb_
{j+1}-\eb_{i-j+1})
+\nonumber\\
&\frac{1}{2} \sum_{i=2}^\infty\sum_{j=1}^{i-1} 
j(i-j)(p_{j}+1)(p_{i-j}+1+\delta_{j,i-j})\,
\Nd_{g-1,k-1,b}(\boldsymbol{p}-\eb_{i+2}+\eb_{j}+\eb_{i-j})+\nonumber\\
&\frac{1}{2} \sum_{i=2}^\infty\sum_{j=1}^{i-1} j(i-j)\sum_{g_1+g_2=g}\;\; 
\sum_{k_1+k_2=k-1}\;\; \sum_{b_1+b_2=b}\;\; 
\sum_{\boldsymbol{p}^{(1)}+\boldsymbol{p}^{(2)}=\boldsymbol{p}-\eb_{i+2}}  
\frac{b!}{b_1! b_2!}\times \nonumber \\
&\hspace{0.8in} (p^{(1)}_{j}+1)(p^{(2)}_{i-j}+1)\,
\Nd_{g_1,k_1,b_1}(\boldsymbol{p}^{(1)}+\eb_{j}) \,
\Nd_{g_2,k_2,b_2}(\boldsymbol{p}^{(2)}+\eb_{i-j})\,.
\label{big_sum_dual}
\end{align}

\label{prop_dual}
\end{prop}
\begin{figure}[hbt]%
 \begin{center}
 \includegraphics[width=8cm]{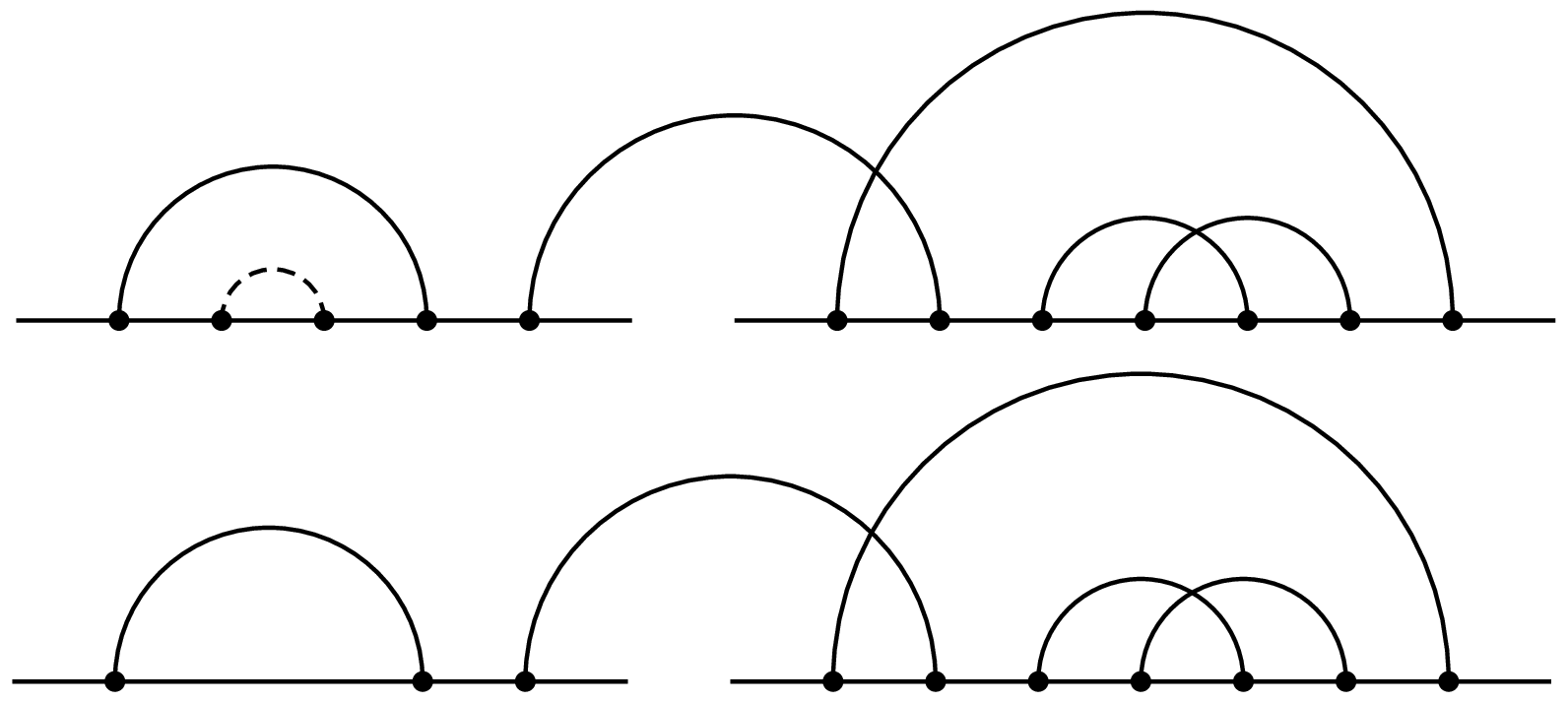}
 \caption{The first backbone has length $6$, and the second one has length $8$. 
Erasing the dashed chord joins two boundary components. }
 \label{cd_a_dual}%
\end{center}
\end{figure}

\begin{figure}[hbt]%
 \begin{center}
 \includegraphics[width=8cm]{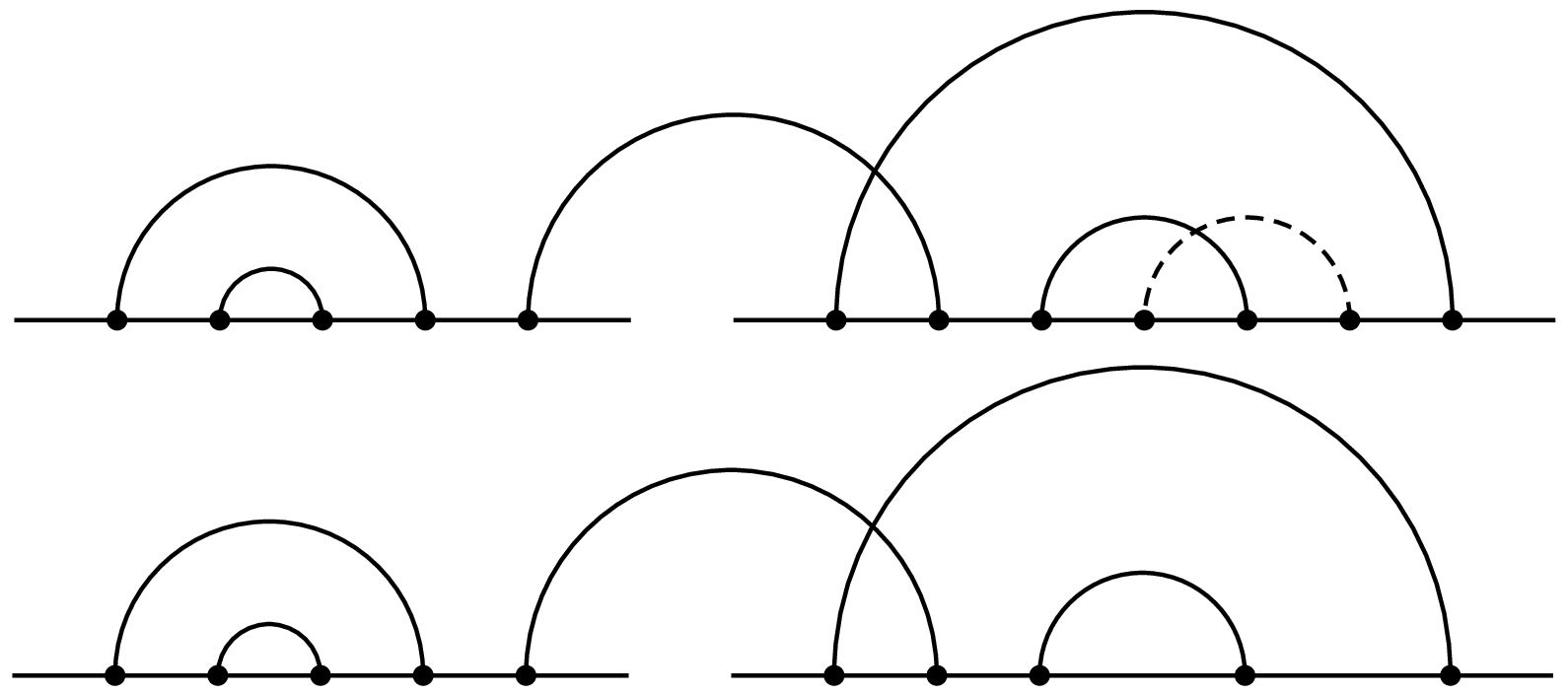}
 \caption{Erasing the dashed chord splits a boundary component into two ones.}
 \label{cd_b_dual}%
\end{center}
\end{figure}

\begin{figure}[hbt]%
 \begin{center}
 \includegraphics[width=8cm]{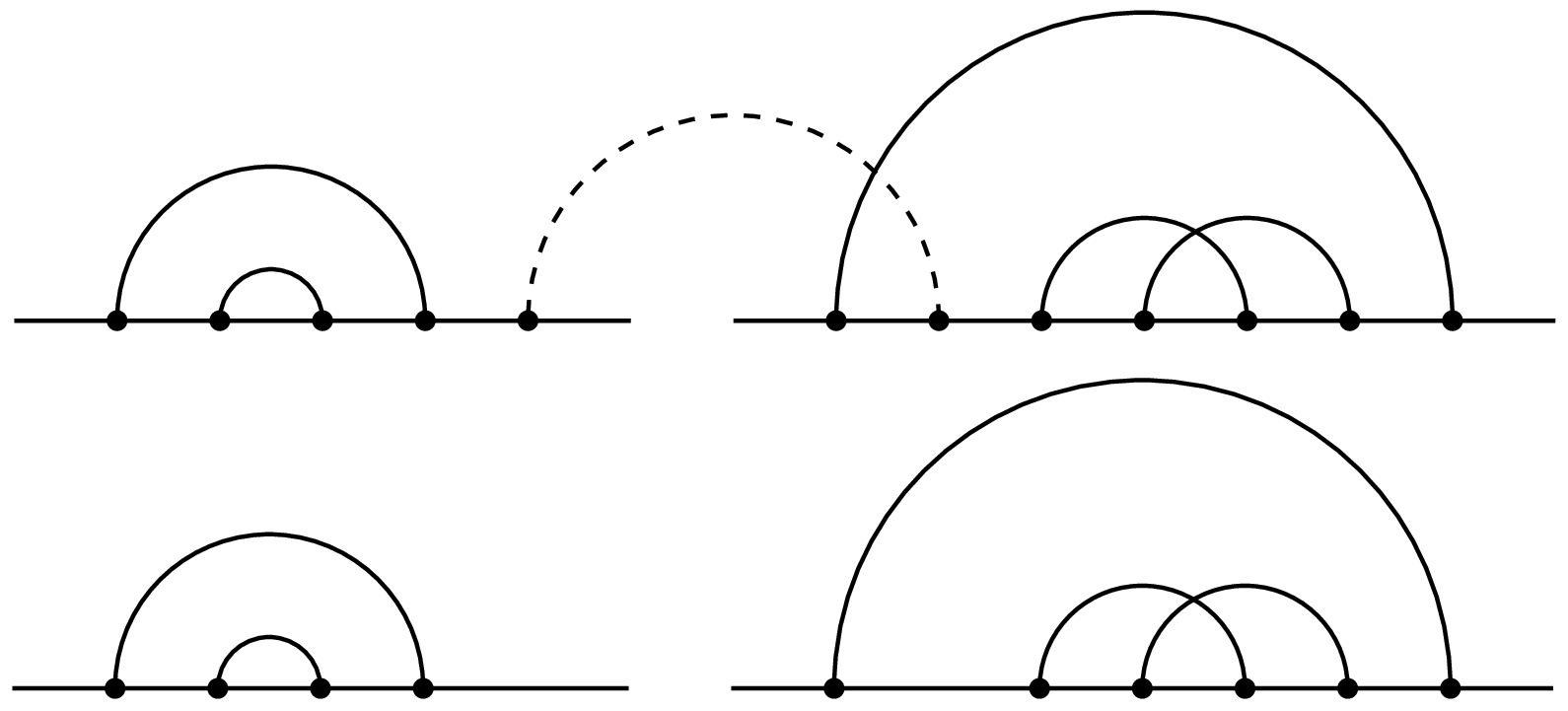}
 \caption{Erasing the dashed chord splits the diagram into two connected 
components.}
 \label{cd_c_dual}%
\end{center}
\end{figure}

\begin{proof}
The proof is similar to that of Proposition \ref{rec_for}. In this case, though, 
we erase a chord together with its endpoints. There are  three possibilities.
The first is that the chord belongs to two distinct boundary components (see 
Fig. \ref{cd_a_dual}). Upon erasing the  chord, these two components join into 
one. This possibility is described by the first term on the r.h.s. of 
(\ref{big_sum_dual}).

The second possibility occurs when the chord belongs to only one boundary 
component. When we erase this chord, the boundary component splits into two  
(see Fig. \ref{cd_b_dual}). In this case, the genus of the diagram decreases by 
$1$, and this is described by the second term on the r.h.s. of 
(\ref{big_sum_dual}).

In these two cases, the diagram remains connected after erasing a chord. The 
third possibility occurs when  erasing the chord splits the diagram into two 
connected components (see Fig. \ref{cd_c_dual}). This yields the third term on 
the r.h.s. of (\ref{big_sum_dual}).

The extension of the proof for the boundary length spectrum in non-orientable 
case follows a logic similar to that for the boundary point spectrum in Section 
\ref{sec:nonor} completing the proof of Theorem \ref{thm2}.
\end{proof}

\section{Matrix integral}\label{secmi}

We show here that certain linear combinations of the numbers 
$\Nd_{g,k,l}(\bb,\nb)$ can be interpreted as integrals over the space of 
Hermitian matrices. Once again, we start with the one backbone  case. Let $P$ be 
a Hermitian $N \times N$ matrix, such that $P^2=P$ and $Tr P = p$. Consider the 
matrix integral 
\begin{align}
M_m(s,p,N)=\int_{\mathcal{H}_N} Tr(X+sP)^m\, d\mu(X),
\end{align}
where ${\mathcal{H}_N}$ is the space of Hermitian matrices and 
$$d\mu(X)=\frac{1}{vol(\mathcal{H}_N)}\exp(-\frac{1}{2} Tr X^2 )\, dX$$ 
is the normalized Gaussian unitary-invariant measure on it (this is a special 
case of a much more general matrix integral considered in \cite{MSh}).
\begin{prop}
We have
\begin{align}
&M_m(s,p,N) =\nonumber\\
&= \sum_{k=0}^{[m/2]} \sum_{g=0}^{[k/2]} \sum_{n_0=0}^{k+1-2g} \sum_{{\sum i n_i 
= m-2k}}\Nd_{g,k,m-2k}(\pmb{e}_{m},\nb)\,s^{m-2k}\, N^{n_0}\,p^{\sum_{i\geq 
1}n_i} \label{mat_int}
\end{align}
\end{prop}

\begin{proof}
We prove (\ref{mat_int}) using the Wick formula.
 
First, note that one can diagonalize the matrix $P$, and this does not change 
the measure $d\mu(X)$. Therefore, one can assume that
$$p_{ij} = \begin{cases}
1, \text{ if } i=j, \quad i \leq p,\\
0, \text{ otherwise.} 
\end{cases}$$ 

Second, note that $M_m(s,p,N)$ is a polynomial in $s$, and the coefficient of 
$s^{m-2k}$ is $\sum_{\alpha,\beta} \int_{\HN} Tr(\Pi_{\alpha,\beta})\,d\mu(X)$, 
where the sum is taken over all products 
$\Pi_{\alpha,\beta}=X^{\alpha_1}P^{\beta_1}\cdots X^{\alpha_m}P^{\beta_m}$ with 
$\alpha_i,\beta_i\in\mathbb{Z}_{\geq 0}$ non-negative integers such that 
$\sum\alpha_i=2k$ and $\sum\beta_i=m-2k$.  We have
$$\int_{\HN} Tr(\Pi_{\alpha,\beta})\,d\mu(X)=\sum_{i_1=1}^N \dots \sum_{i_m=1}^N 
\int_{\HN} y_{i_1i_2}y_{i_2i_3}\dots y_{i_m i_1}\,d\mu(X),$$
where 
$$y_{i_ji_{j+1}} = \begin{cases}
x_{i_ji_{j+1}}, \text{ if $X$ is the $j$-th factor in the product } 
\Pi_{\alpha,\beta},\\
p_{i_ji_{j+1}}, \text{ if $P$ is the $j$-th factor in the product } 
\Pi_{\alpha,\beta}.
\end{cases}$$ 
To compute the expectation of the product $y_{i_1i_2}y_{i_2i_3}\dots y_{i_m 
i_1}$, one has to count all possible matchings between indices of the 
$X$-factors.  Any product with such a matching can be graphically represented by 
a chord diagram with $k$ chords and $m-2k$ marked points on the backbone, where 
the chords correspond to the matched $X$-factors, and the marked points 
correspond to $P$-factors. Each boundary component of the chord diagram is 
therefore labeled by some index $i_j$. If there are no marked points on the 
boundary component, then the corresponding index $i_j$ can take any value from 
$1$ to $N$. If there are marked points on the boundary component, then the 
corresponding index $i_j$ can only take values from 1 to $p$, because $p_{ii}$ 
is nonzero only when $i\leq p$. Thus, we have
\begin{align*}
\sum_{\alpha,\beta} \int_{\HN} Tr(\Pi_{\alpha,\beta})\,d\mu(X)=
\sum_{g=0}^{[k/2]} \sum_{{\sum i n_i = m-2k}}  \Nd_{g,k,m-2k}(\pmb{e}_{m},\nb) 
N^{n_0}p^{\sum_{i\geq 1}n_i}\;
\end{align*}
which completes the proof.
\end{proof}

Let us take the one backbone generating function 
$F_1(x,y;s_0,s_1,\dots;t_1,t_2,\dots)$  given by (\ref{gf}) and put $x=1$, 
$y=1/z^2$, $s_0=N$, $s_i=ps^i/z^i,\; i=1,2,\dots$, $t_j=1,\;j=1,2,\dots$. This 
gives us the expectation of the resolvent of $X+sP$:
\begin{align}
F_1&\left(1,\frac{1}{z^2},N;\frac{ps}{z},\frac{ps^2}{z^2},\dots;1,1,
\dots\right)=\nonumber\\
=&\sum_{m=0}^\infty \sum_{k=0}^{[m/2]} \sum_{g=0}^{[k/2]} \sum_{n_0=0}^{k+1-2g} 
\sum_{{\sum i n_i = m-2k}} 
\Nd_{g,k,m-2k}(\pmb{e}_{m},\nb)\,z^{-m}\,N^{n_0}\,p^{\sum_{i\geq 
1}n_i}\nonumber\\
=&-z\int_{\mathcal{H}_N} Tr(X+sP-zI)^{-1} \,d\mu(X)\;.\label{res}
\end{align}

\smallskip\noindent
{\bf Non-orientable case.} The numbers $\widetilde\Nd_{h,k,l}(\bb,\nb)$ appear 
as coefficients in the expansion of a matrix integral similarly to  the numbers 
$\Nd_{g,k,l}(\bb,\nb)$
again by Wick's Theorem. Namely, consider the matrix integral
\begin{align}
K_m(s,p,N)=\int_{\mathcal{H}_N(\mathbb{R})} Tr(X+sP)^m\, d\nu(X),
\end{align}
where ${\mathcal{H}_N}(\mathbb{R})$ is the space of real symmetric matrices and 
$$d\nu(X)=\frac{1}{vol(\mathcal{H}_N(\mathbb{R}))}\exp(-\frac{1}{2} Tr X^2 )\, 
dX$$ 
is the normalized Gaussian orthogonal-invariant measure on it.
\begin{prop}
We have
\begin{align}
&K_m(s,p,N) =\nonumber\\
&= \sum_{k=0}^{[m/2]} \sum_{h=0}^{[k/2]} \sum_{n_0=0}^{k+1-h} \sum_{{\sum i n_i 
= m-2k}}   
\widetilde\Nd_{h,k,m-2k}(\pmb{e}_{m},\nb)\,s^{m-2k}\,N^{n_0}\,p^{\sum_{i\geq 
1}n_i} \label{mat_int_non}
\end{align}
\end{prop}

\smallskip\noindent
{\bf Multibackbone case.} In the multibackbone case, the matrix integral has a 
similar form. Take a sequence $\bb=(b_1,b_2,\dots)$ with a finite number of 
non-zero elements that are positive integers. 
 Consider the matrix integral
 $$I_N(s,p,\bb)=\int_{\mathcal{H}_N} \prod_{m=0}^\infty (Tr(X+sP)^m)^{b_m} \, 
d\mu(X).$$
 This integral is related to the total generating function $F$ by the formula:
 \begin{align}
\frac{1}{b!}\,I_N(s,p,\bb)=\sum_{k=0}^{\infty}\;\sum_{{\sum i n_i = \sum i 
b_i-2k}}\widehat{\Nd}_{k,g,m-2k}(\bb,\nb)N^{n_0}\,s^{\sum i 
b_i-2k}\,p^{\sum_{i\geq 1}n_i},
 \end{align}
where $\widehat{\Nd}_{k,g,m-2k}(\bb,\nb)$ is the coefficient of $x^{2g-2} y^k 
s_0^{n_0}s_1^{n_1}\dots t_1^{b_1}t_2^{b_2}\dots$ in the power series expansion 
of $e^F, \;b=\sum_{i\geq 1} b_i$.

In \cite{Kaz96}, there is a matrix integral interpretation for the numbers 
$\Nd_{0,k,b}(\boldsymbol{p})$.
\linebreak
Namely, let
$${\mathcal F}({\pmb{ s}},t) = \sum_{\boldsymbol{p}}\sum_{b=1}^\infty 
\Nd_{0,k,b}(\boldsymbol{p}) {\pmb{ s}}^{\boldsymbol{p}} {t}^{b},$$
then one has 
$$\log \int_{\mathcal{H}(N)} \exp \left(-\frac{N}{2} Tr X^2 + \sum_{k=1}^\infty 
\frac{t^k}{k} Tr (XA)^k\right) dX \to {\mathcal F}({\pmb{ s}},t), $$
where $A$ is some matrix such that $s_i = \frac{1}{iN}\, Tr A^i$.

For the large $N$ limit in the $1$-backbone case,  this matrix integral can be 
modified so that the limit distribution is computable by free probability 
methods. Namely, consider semi-positive definite matrix $AXA^* A X^* A^*$. For 
any integer $k>0$, we have
\begin{align}
\lim_{N \to \infty}\frac{1}{N}\int Tr  \left(AXA^*(AXA^*)^* \right)^k 
\exp\left(-\frac{N}{2}TrXX^*\right) dX = 
\sum_{\boldsymbol{p}}\Nd_{0,k,b}(\boldsymbol{p}) {\pmb{ s}}^{\boldsymbol{p}},  
\label{mm}
\end{align}
where $s_i = \frac{1}{N} Tr(AA^*)^i$.

\smallskip\noindent
{\bf Asymptotic spectral distribution.} In the one backbone case, we can compute 
the leading term  in the asymptotics of the matrix integral. To treat the large 
$N$ limit of (\ref{res}), one can use the techniques of free probability. Put 
$p=[qN]$ with some $q \in (0,1)$ and consider the limit 
\begin{align}
\widetilde{M}_m(s,q)=\lim_{N \to \infty}{N^{-m/2-1} M_m(s\sqrt{N},p,N)}. 
\end{align}
This limit is a polynomial in $q$ and $s$, and the coefficient at $s^l q^n$ is 
the number of chord diagrams with $\frac{m-l}{2}$ chords, $l$ marked points and 
$n$ boundary components containing at least one marked point (i.e., 
$n=\sum_{i\geq 1} n_i$). Note that $\widetilde{M}_m(s,q)$ are the moments of a 
probability measure on $\mathbb{R}$, namely, the limit spectral distribution of 
the matrices $X/\sqrt{N}+s P$. This measure is uniquely determined by the limit 
spectral measures of $X/\sqrt{N}$ and $sP$. 

Let us define the $\mathcal{R}$-transform $\mathcal{R}_\mu (z)$ and the 
$\mathcal{S}$-transform $\mathcal{S}_\mu (z)$ of a measure $\mu$.
We start with the moment generating function $\mathcal{M}_{\mu}(z)$ and the 
Cauchy transform $G_{\mu}(z)$ which are defined by the series
\begin{align}
G_{\mu}(z)=\sum_{m=0}^\infty \frac{M_m}{z^{m+1}},\label{ca}
\end{align}
and 
\begin{align}
\mathcal{M}_{\mu}(z)=\sum_{m=1}^\infty {M_m}{z^{m}},\label{mgf}
\end{align}
where $M_m=\int_{\mathbb{R}}x^m d\mu(x)$ are the moments of the measure $\mu$.
The (unique) solution of the equation
\begin{equation}
\mathcal{R}_\mu (G_\mu(z))+\frac{1}{G_\mu(z)}=z.
\label{rtr}
\end{equation}
is  $\mathcal{R}_\mu (z)$.
The $\mathcal{S}$-transform is defined by
\begin{equation}
\mathcal{S}_\mu(z)= \frac{z+1}{z}\mathcal{M}^{-1}_\mu(z).
\label{str}
\end{equation}

The following is standard \cite{Speicher}:
\begin{prop}\label{fp}
(1) If $A_N$ and $B_N$ are two random  Hermitian $N\times N$ matrices in general 
position, and the limit spectral distributions of $A_N$ and $B_N$ are $\mu$ and 
$\nu$ respectively, then the limit spectral distribution of $A_N+B_N$ is some 
distribution  $\mu \boxplus \nu$, which is determined by its 
$\mathcal{R}$-transform 
$$ \mathcal{R}_{\mu \boxplus \nu} (z) = \mathcal{R}_{\mu} (z)+\mathcal{R}_{\nu} 
(z).$$

(2) If $A_N$ and $B_N$ are two random  $N\times N$ matrices in general position, 
and the limit spectral distributions of $A_N A_N^*$ and $B_N B_N^*$ are $\mu$ 
and $\nu$ respectively, then the limit spectral distribution of 
$A_NB_N(A_NB_N)^*$ is some distribution  $\mu \boxtimes \nu$, which is 
determined by its $\mathcal{S}$-transform 
$$ \mathcal{S}_{\mu \boxtimes \nu} (z) = \mathcal{S}_{\mu} (z)\mathcal{S}_{\nu} 
(z).$$
\end{prop}

Thus, if one knows the $\mathcal{R}$-transform of the spectral distribution of 
$X/\sqrt{N}$ (let it be $\mu$) and of $sP$ (let it be $\nu$), then one also 
knows the $\mathcal{R}$-transform of the spectral distribution of 
$X/\sqrt{N}+sP$. Computing the Cauchy transform of the latter and expanding it 
in the inverse powers of $z$, one gets the coefficients $\widetilde{M}_m(s,q)$ 
in accordance with (\ref{ca}).
Note that the measure $\mu$ appears in the famous Wigner semicircle law, i.e.,
$$d \mu(x) =\begin{cases}
\frac{1}{2\pi}\sqrt{4-x^2}, \text{ if } -2\leq x\leq2\\
0, \text{ otherwise,} 
\end{cases}$$
and the measure $\nu$ is a  two-point distribution 
$q\delta(x-1)+(1-q)\delta(x)$.
Now we compute the Cauchy transforms:
\begin{align*} 
&G_{\mu}(z)=\frac{z-\sqrt{z^2-4}}{2},\\
&G_{\nu}(z)=\frac{q}{z-s}+\frac{1-q}{z}=\frac{z-s+qs}{z(z-s)}
\end{align*}
and explicitly solve the equation (\ref{rtr}):
\begin{align*} 
&\mathcal{R}_{\mu} (z)=z,\\
&\mathcal{R}_{\nu} (z) = \frac{zs-1+\sqrt{(zs-1)^2+4zsq}}{2z}. \end{align*}
Thus, the Cauchy transform $G(z)$ of the limit spectral measure of the matrix 
$X/\sqrt{N}+sP$ satisfies the equation
$$
G(z) + s +\frac{\sqrt{(sG(z)-1)^2-4sqG(z)}}{2G(z)}=z.
$$
This allows us to compute the first several polynomials $\widetilde{M}_m(s,q)$, 
and we find that they coincide with previous computations: 
$\widetilde{M}_m(s,q)$ is the coefficient of $y^m$ found in Remark 
\ref{rem_coeff} by purely combinatorial methods.

Let us use the $\mathcal{S}$-transform technique to compute the limit spectral 
distribution of a random matrix $AXA^*AX^*A^*$, where $X$ is a standard complex 
Gaussian $N\times N$ matrix with variance $N^{-1/2}$, and $A$ is an $N\times N$ 
matrix such that $\lim_{N\to \infty} \frac{1}{N} Tr (AA^*)^{k}=s_k$. In other 
words, $s_k$ are the moments of the limit spectral distribution of $AA^*$  that 
we denote by $\nu$. Its $\mathcal{S}$-transform is given by
$$
\mathcal{S}_{\nu}(z)=\frac{z+1}{z}\,{\mathcal M}_{\nu}^{-1}(z)\;,
$$
where
$$
{\mathcal M}_{\nu}(z)=\sum_{k=1}^{\infty}s_kz^k 
$$
and ${\mathcal M}_{\nu}^{-1}$ is the inverse function to ${\mathcal M}_{\nu}$.
The limit spectral distribution of $XX^*$ (let it be $\mu$) is the 
Marchenko--Pastur distribution with parameter $1$ and has 
$\mathcal{S}$-transform of the form 
$$\mathcal{S}_{\mu}(z)=\frac{1}{1+z}.$$
The limit spectral distribution $\lambda$ of $AXA^*AX^*A^*$ therefore has 
$\mathcal{S}$-transform 
$$\mathcal{S}_{\lambda}(z)=\frac{1}{1+z}\,\mathcal{S}^2_{\nu}(z).$$

The previous equation allows to compute length spectra for planar diagrams on 
one backbone, namely:

\begin{thm} Put ${\mathcal K}(z)=\frac{z}{1+z}\, {\mathcal S}_{\lambda}(z)$. 
Then the one backbone generating function $G_1(x,z;{\pmb{s}})$ for boundary 
length spectra in genus zero is given by
$$G_1(0,z;{\pmb s})=1+ {\mathcal K}^{-1}(z).$$  In particular, we have 
$G_1(0,z;1,1\ldots )=C_{0,1}(z)$, the Catalan
generating function.
\end{thm}

The proof of this theorem immediately follows from (\ref{mm}), (\ref{str}) and 
Proposition \ref{fp}, (2). To check that ${\mathcal K}^{-1}(z)$ generates the 
Catalan numbers, we notice that for $s_k=1$ for all $k$ we have ${\mathcal 
M}_{\nu}(z)=\frac{z}{1-z}$ and ${\mathcal M}_{\nu}^{-1}(z)=\frac{z}{1+z}$. 
Therefore, ${\mathcal S}_{\nu}(z)=1$ and ${\mathcal 
S}_{\lambda}(z)=\frac{1}{1+z}$. Thus we have ${\mathcal K}(z)=\frac{z}{(1+z)^2}$ 
and
$$
{\mathcal K}^{-1}(z)=\frac{1-2z-\sqrt{1-4z}}{2z}
$$
(since ${\mathcal K}^{-1}(0)=0)$, that is a well-known generating function  for 
the Catalan numbers.

\section{Closing Remarks}\label{closing}

See \cite{PKWA1,PKWA2} for an application of the non-orientable diagrams to 
modeling the topology of proteins.

Inspired by the results of this paper and with an eye to understanding the 
multibackbone analog  of the differential equation equivalent to the 
Harer-Zagier recursion 
$(n+1)C_{g,1,n}=(2n-1)\left( 2C_{g,1,n-1}+\binom{2n-2}{2} C_{g-1,1,n-2} 
\right)$, we ask if there is a differential operator in the variables $(x,y,t)$ 
that vanishes on $C(x,y,t)$ and thus determines it.  In fact, the Master Loop 
Equation of the model in \cite{ACPRS} provides a constraint on $C(x,y,t)$ that 
however fails to give a differential operator.

{\bf Acknowledgements.} We thank M.~Kazarian for suggesting a link between 
Theorems 1, 2 and the KP theory.

\end{document}